\documentclass[11pt,a4paper]{article}
\usepackage[utf8]{inputenc}
\usepackage{amsmath}
\usepackage{amsfonts}
\usepackage{amssymb}
\usepackage{amsthm}
\usepackage{enumerate}
\usepackage[dvipsnames]{xcolor}
\usepackage[bookmarks=false]{hyperref}
\usepackage{tikz-cd}
\usepackage{authblk}

\hypersetup{colorlinks,linkcolor={red!70!black},citecolor={blue!50!black},urlcolor={red!50!black}}
\author[1]{Avijit Panja\thanks{Email: avijit.0050@gmail.com - Avijit Panja is partially supported by the DST-RCN grant INT/NOR/RCN/ICT/P-03/2018 from the Dept. of Science and Technology, Govt. of India.}}
\author[2]{Rakhi Pratihar\thanks{Email: pratihar.rakhi@gmail.com - During this course of work, Rakhi Pratihar was supported by a doctoral fellowship at IIT Bombay from the University Grant Commission, Govt. of India (Sr. No. 2061641156). Currently, she is supported by Grant 280731 from the Research Council of Norway.}}
\author[3]{Tovohery Hajatiana Randrianarisoa \thanks{Email: tovo@aims.ac.za - During this course of work, Tovohery Randrianarisoa was supported by a postdoc fellowship at IIT Bombay from the Swiss National Science Foundation Grant No. 181446.}}
\affil[1]{Department of Mathematics, Indian Institute of Technology Bombay, Powai, Mumbai 400076, India}
\affil[2]{Department of Mathematics and Statistics, UiT - The Arctic University of Norway, N-9037, Tromsø, Norway}
\affil[3]{Department of Mathematical Sciences, Florida Atlantic University, Boca Raton, FL 33431, USA}

\title{Some Matroids Related to Sum-Rank Metric Codes}
\date{\vspace{-5ex}}
\newtheorem{thm}{Theorem}
\newtheorem{pro}[thm]{Proposition}
\newtheorem{cor}[thm]{Corollary}
\newtheorem{lem}[thm]{Lemma}   
\theoremstyle{definition}            
\newtheorem{defn}[thm]{Definition}
\newtheorem{rem}[thm]{Remark}
\newtheorem{exa}[thm]{Example}

\newcommand{\0}{\mathbf{0}}
\newcommand{\nn}{\eta}
\newcommand{\Y}{\mathbf{Y}}
\newcommand{\Rk}{\mathrm{Rk}}
\newcommand{\F}{\mathbb{F}}
\newcommand{\K}{\mathbf{K}}

\renewcommand{\L}{\mathcal{L}}
\newcommand{\C}{\mathcal{C}}
\newcommand{\D}{\mathcal{D}}
\newcommand{\G}{\mathbf{G}}
\renewcommand{\H}{\mathbf{H}}
\newcommand{\N}{\mathbb{N}}
\renewcommand{\c}{\mathbf{c}}

\newcommand{\PKn}{\P(\K^{\n})}
\newcommand{\n}{\mathbf{n}}
\newcommand{\Fq}{\mathbb{F}_q}
\newcommand{\Fqm}{\mathbb{F}_{q^m}}
\newcommand{\rank}{\mathbf{rank}}
\newcommand{\p}{\perp}
\newcommand{\<}{\left<}
\renewcommand{\>}{\right>}
\newcommand{\sr}{\mathbf{srank}}

\newcommand{\x}{\mathbf{x}}
\newcommand{\y}{\mathbf{y}}

\newcommand{\M}{\mathbf{M}}
\renewcommand{\P}{\mathcal{P}}

\newcommand{\e}{\emph}
\renewcommand{\AA}{\mathbf{A}}

\newcommand{\diag}{\mathrm{diag}}

\newcommand{\Nu}{\mathcal{V}}
\renewcommand{\l}{\lambda}

\newcommand{\avijit}[1]{{\color{blue} #1}}
\newcommand{\tovo}[1]{{\color{PineGreen} #1}}

\begin{document}

\maketitle

\begin{abstract}
We introduce the notion of sum-matroids and show its association with sum-rank metric codes. As a consequence, some results for sum-rank metric codes by Mart{\'i}nez-Pe{\~n}as are generalized for sum-matroids. The sum-matroids generalize the notions of matroids and $q$-matroids. We define the generalized weights for sum-matroids and prove a Wei-type duality theorem which generalizes the analogous cases for matroids and q-matroids.
\end{abstract}
\section{Introduction}\label{Sec:1}
The sum-rank metric, a common generalization of the Hamming and the rank metric, was first introduced implicitly for space-time coding \cite{EH03,LK05} and then it was used in measuring error-correcting capability of codes in multishot network coding \cite{NU10}. Considering its practical importance, several sum-rank metric codes have been constructed in \cite{MP2018,NPRV17,NU10,PR19,UMP2019}. On the other hand, the theoretical study of support spaces corresponding to sum-rank metric was initiated in \cite{MP2019} and thereafter, the fundamental properties of codes embedded with the sum-rank metric are studied in \cite{BGR21}. 

For a positive integer $l$, let $(n_1,\ldots, n_l)$ and $(m_1,\ldots,m_l)$ be $l$-tuples of positive integers and $n = \sum\limits_{i=1}^l n_i$. For $i = 1, \ldots, l$, let $K_i$ be finite fields with a common extension $\F$ such that $\F / K_i$ is a finite extension of degree $m_i$. Then a $k$-dimensional $\F$-linear subspace $\C$ of the $n$-dimensional vector space $\F^n$ is called an $[n,k;n_1,\ldots,n_l]$-sum-rank metric code over $(\F: K_1,\ldots,K_l)$, where the \emph{sum-rank distance} between two codewords is defined as follows. Fix a $K_i$-basis $\Gamma_i$ of $\F$ so as to associate to any vector $\c^{(i)}\in \F^{n_i}$ an $m_i\times n_i$ matrix $\Gamma_i(\c^{(i)})$ with entries in $K_i$, for $0 \leq i \leq l$. Now the sum-rank distance between any $\x=(\x^{(1)},\ldots,\x^{(\ell)})$, $\y=(\y^{(1)},\ldots,\y^{(\ell)}) \in \F^n$, with $\x^{(i)}, \y^{(i)} \in \F^{n_i}$, is defined as the sum of the rank of the matrices $\Gamma_i(\x^{(i)} -\y^{(i)})$. For the particular case of $n_i = 1$ for $0 \leq i \leq l$, the sum-rank distance equals the Hamming distance \cite{Ham50, MS78} while the case $l=1$ recovers the notion of rank distance \cite{Del78, Gab85,Rot91}. 

The theory of linear Hamming metric codes has been studied in a more general context of matroids \cite{Ox06}. In 1976, Greene first mentioned several connections between matroid and coding theory \cite{Gre76} and thereafter, many authors demonstrated the importance of this connection by proving coding-theoretical results using matroid theory, e.g, see \cite{Bri07,BJM14,JRV16,JV13}. Fundamental results like descriptions of support of a codeword and higher supports of subcodes in various terms of the corresponding matroid can be found, for example, in \cite{Bri07,Gre76,Whit92}. Also, Britz, Johnsen, Mayhew and Shiromoto in \cite{BJMS12} generalized the celebrated Wei duality theorem \cite{Wei91} for matroids. Inspired by Wei's work in \cite{Wei91} about how generalized weights characterize security in wire-tap channel of type II, Kurihara et al. established the relevance of generalized rank weights for secure network coding \cite{KMU15}. As in the case of linear Hamming metric codes, generalized (rank) weights for rank metric codes also have been studied in more general contexts. The $q$-analogues of matroids, called $q$-matroids and $q$-polymatroids, have been considered by Jurrius and Pellikaan in \cite{JP2016} and by Gorla, Jurrius, Lopez, and Ravagnani in \cite{GJLR19}, respectively. Shiromoto has introduced the notion of $(q,m)$-polymatroid \cite{Shiro19} and recently, Ghorpade and Johnsen \cite{GJ19} and Britz, Mammoliti, and Shiromoto \cite{BMS19} have independently proved a Wei-type duality theorem for rank metric codes for more general combinatorial structures, called $(q,m)$-demi-polymatroids.
 
So it is natural to ask for a combinatorial structure such that coding-theoretic results for sum-rank metric codes can be studied in a more general context. It is also expected that the matroid analogue related to sum-rank metric code will generalize matroids and $q$-matroids. This is the question we address in this paper.
We introduce the notion of \emph{sum-matroids} so as to have a common generalization of the notions of matroids and $q$-matroids. By associating sum-matroids with sum-rank metric codes, we recover several results of Mart{\'i}nez-Pe{\~n}as on the theory of sum-rank metric codes. In particular, we generalize the notion of generalized weights for sum-rank metric codes to sum-matroids. Moreover, we prove a Wei-type duality theorem for the generalized weights of sum-matroids that combines the existing Wei-type duality theorems of linear codes and matroid like structures.

Another motivation of introducing sum-matroids is from the set of works \cite{GPR21,JV13,JPV21} where the authors associate to a linear code with Hamming metric (resp. rank metric), or more generally, to a matroid (resp. $q$-matroid), a fine set of invariants, called Betti numbers (resp. Virtual Betti numbers and singular homology). For Hamming metric codes, these are obtained by considering a minimal free resolution of the Stanley-Reisner ring corresponding to a matroid associated to the linear code \cite{JV13}. On the other hand, for rank metric codes the virtual Betti numbers \cite{JPV21} are defined as certain M\"obius numbers associated to the lattice of flats of the corresponding $q$-matroid and in \cite{GPR21}, the singular homology of the poset (equipped with the order topology) of independent spaces of the corresponding $q$-matroid has been computed. One can expect to generalize these results for sum-matroids by considering the associated generalized lattice structures and posets. Nonetheless, we do not consider these questions in this paper.

The paper is organized as follows. In Section \ref{Sec:2}, we collect some preliminaries concerning linear codes and matroid-like structures and outline some relevant basic notions and results. Here we include a new characterization of {\em maximum sum-rank distance} (MSRD) codes. The notion of sum-matroids is introduced in Section \ref{Sec:3}. We associate sum-matroids to sum-rank metric codes in Section \ref{Sec:4}. The generalized weights of sum-matroids are defined and a Wei-type duality theorem for them is established in Section \ref{Sec:5}. As a corollary, one recovers the Wei-type duality theorems for sum-rank metric codes, $q$-matroids and matroids. Finally, in Section \ref{Sec:6} we conclude by mentioning some further directions for studying sum-matroids.

\section{Preliminaries}\label{Sec:2}

For once and for all, we fix a positive integer $l$ and two $l$-tuples of positive integers $(n_1,\ldots,n_l)$ and $(m_1,\ldots,m_l)$. Let $n =\sum\limits_{i=1}^{l} n_i$ and  for $0 \leq i \leq l$, $K_i$'s are finite fields with a common extension $\F$ such that $[\F:K_i] = m_i$. Let $\Gamma_i$ be a $K_i$-basis of $\F$ for $0 \leq i \leq l$. We use $\n$ and $\K$ to denote the $\ell$-tuples $(n_1,\ldots,n_{\ell})$ and $(K_1,\ldots,K_{\ell})$, respectively and $\K^{\n}$ to denote the $\ell$-tuple of vector spaces $(K_1^{n_1},\ldots,K_{\ell}^{n_\ell})$. 
Throughout this paper, whenever we speak of an $[n,k;n_1,\ldots,n_l]$-sum-rank metric code, we mean the code to be over $(\F:K_1, \ldots,K_l)$ as defined in the Introduction. We use $\Fq$ to denote the finite field with $q$ elements, where $q$ is a prime power.

This section is divided into two parts. In Subsection \ref{Sec:2.1} we review the theory of supports and generalized weights of linear codes. We explicitly mention the notions and results for sum-rank metric codes, whereas the analogous notions and results for Hamming and rank metric codes are remarked as particular cases. In Subsection \ref{Sec:2.2}, we collect the corresponding notions and results about matroids and $q$-matroids.

\subsection{Linear codes}\label{Sec:2.1}
\begin{defn}
A linear code $\C$ of length $n$ and dimension $k$ over $\F$ is a $k$-dimensional $\F$-linear  subspace of $\F^n$. We call $\C$ an $[n,k]$-code.
\end{defn}

From now on, whenever we write ``code'' over $\F$, we mean ``linear code'' over $\F$. A subcode $\D$ of $\C$ is simply an $\F$-linear subspace of $\C$.

Here we shall see how the sum-rank metric defined in Introduction generalizes the Hamming and the rank metric.

First we recall from \cite[Definition 1]{MP2019} the formal definition of the sum-rank weight of a vector $\c \in \F^n$ as follows. For $\c = (c^{(1)},\ldots, c^{(l)}) \in  \F^n$, its sum-rank weight is   
\begin{equation}\label{sumrankweight}
\sr(\c)= \sum \limits_{i=1}^{\ell}  \text{rank}_{K_i}(\Gamma_i(\c^{(i)})),
\end{equation} where $\Gamma_i(\c^{(i)})$ is the coordinate matrix of $c^{(i)} \in \F^{n_i}$ w.r.t. the fixed basis $\Gamma_i$ for all $0 \leq i \leq l$. 

%In this section we consider the following metrics.

Assume that $\x = (x_1,\ldots,x_n)$ and $\y = (y_1,\ldots,y_n)$ are elements of $\F^n$. Then the sum-rank distance between $\x$ and $\y$, given by 
\begin{equation}\label{srd}
    d_{SR}(\x,\y):=\sr(\x-\y)
\end{equation}
defines a metric on $\F^n$, which is called the sum-rank metric \cite{NU10}.

Now it is easy to see that, for $\n = (1,\ldots,1)$, the sum-rank distance is the same as the Hamming distance defined as follows
\begin{equation}\label{hd}
    d_H(\x,\y):= |\{i\colon x_i \neq y_i, 1\leq i\leq n\}|.
\end{equation}
	The $[n,k]$-codes over $\F$ endowed with this metric are called Hamming metric codes over $\F$ \cite{Ham50,MS78}.
The case $l=1$ (say, $\K=K_1$) recovers the rank distance defined as follows. 	
\begin{equation}\label{rd}
d_R(\x,\y):= \dim_{K_1}\<x_i - y_i, 1\leq i\leq n\>=\text{rank}_{K_1}(\Gamma_1(\x) - \Gamma_1(\y)),	 
\end{equation}
where $\<x_i - y_i, 1\leq i\leq n\>$ is the $K_1$- linear subspace of $\F$ generated by the $(x_i - y_i)$'s.
Linear codes over $\F$, where $\F^n$ is endowed with above rank metric are called (vector) rank metric codes \cite{Gab85,Rot91}. For detailed discussion on rank-metric one can refer to the book chapter \cite{Gor19}.

Now we recall from \cite{MP2019} the following definitions related to support spaces corresponding to sum-rank metric.

\begin{defn}\cite[Definition 2]{MP2019}\label{def@7}
The cartesian product lattice is defined as
\[
\P(\K^{\n})=\P(K_1^{n_1}) \times \cdots \times \P(K_{\ell}^{n_{\ell}}),
\]
where $\P(K_i^{n_i})$ is the lattice of $K_i$-linear subspaces of $K_i^{n_i}$ for all $i=1,\ldots,\ell$.
\end{defn}

 This set $\P(\K^{\n})$ forms a lattice with componentwise inclusion as the partial order. For any two elements $\L,\, \L' \in \PKn$, their meet $\L \cap \L'$, join $\L + \L'$, and complement $\L^\perp$ are defined by componentwise intersection, sum and orthogonal complements, respectively. As the elements in $\PKn$ are tuples of vector spaces, we can define direct sum of two elements $\L, \, \L'$ by taking the direct sum componentwise.

\begin{defn}\cite[Definition 3]{MP2019}
For $\L=(\L_1,\ldots,\L_{\ell})$ in $\P(\K^{\n})$,

 \[\text{Rk}(\L):= \sum \limits_{i=1}^{\ell} \text{dim}_{K_i}(\L_i).\]
\end{defn}

The following definitions generalize the notion of support spaces as defined for codes with Hamming and rank metrics. 
%Later we shall define the generalized sum-rank weights using support spaces.

\begin{defn}\cite[Definition 4]{MP2019}  
Let $\c=(\c^{(1)},\ldots,\c^{(\ell)}) \in \F^n$,  where $\c^{(i)} \in \F^{n_i}$, for $i=1,\ldots,\ell$.  The \emph{sum-rank support} of $\c \in  \F^n$ is defined as the $\ell$-tuple
\[
\text{supp}(\c)=(\L_1, \ldots ,\L_{\ell}) \in \P(\K^{\n}),
\] 
where $\L_i$ is the $K_i$-linear row space of $\Gamma_i(\c^{(i)})$ for $i=1,\ldots,\ell$.
\end{defn}

\begin{rem}
Note that 
\[
\sr(\c)= \sum \limits_{i=1}^{\ell}\text{rank}_{K_i}(\Gamma_i(\c^{(i)})) = \sum \limits_{i=1}^{\ell} \text{dim}_{K_i}(\L_i)=\Rk(\text{supp}(\c)).
\]
\end{rem}

\begin{defn}\cite[Definition 6]{MP2019} 
Given an $\F$-linear  subspace $\D \subset \F^n$,  its \emph{sum-rank support}  is defined as
\[
\rm {supp}(\D)= \sum \limits_{\c \in \D}\rm{supp}(\c) \in \P(\K^{\n}).
\]
The \emph{sum-rank weight} of $\D$ is $\text{wt}_{SR}(\D):=\Rk(\rm{supp}(\D))$.
\end{defn}

\begin{defn}\cite[Definition 5]{MP2019}\label{def@11}
Let $\L \in \P(\K^{\n}).$ The \emph{sum-rank support space} in $\F^n$ associated to $\L$ is defined as  the vector space
\[
\Nu_{\L}=\{\c \in \F^n: \rm{ supp }(\c) \subseteq \L\},
\]
where $\subseteq$ denotes the componentwise inclusion.
\end{defn}

Note that $\Nu_{\L}$ is an $\F$-linear subspace of $\F^n$. Moreover, as proved in \cite[Corollary 1]{MP2019},
\begin{equation}\label{eq:::2}
 \dim_{\F} \Nu_{\L}  = \Rk(\L).
\end{equation}

\begin{defn}\cite{MP2019}\label{def::12}
Let $\C$ be an $[n,k; n_1,\cdots, n_l]$-sum-rank metric code. The \emph{$i$-th generalized sum-rank weight} of $\C$ is defined as follows.
\[
d_{SR,i}(\C) = \min \{\Rk(\L) : \L \in \P(\K^{\n}) \text{ and} \;  \dim(\C \cap \Nu_{\L}) \geq i\},
\]
for $i=1, \ldots, k$.  
\end{defn}
The following result can be seen as a consequence of \cite[Proposition 4]{MP2019}. We include a proof as it makes it clear that the notion of generalized weights of sum-rank metric codes generalizes those of Hamming and rank-metric codes.
\begin{lem}\label{genweight}
Let $\C \subseteq \F^n$ be an $[n,k; n_1,\cdots, n_l]$-sum-rank metric code with $\text{dim}_{\F}(\C)=k\geq 1$. Then 
\[d_{SR,i}(\C)=\min \{{\text{wt}_{SR}}(\D) \colon \D \subseteq \C \text{ and } \dim_{\F}\D = i  \}.\]
\end{lem}

\begin{proof}
    Let $\min \{{\text{wt}_{SR}}(\D) \colon \D \subseteq \C \text{ and } \dim_{\F}\D = i  \}={\text{wt}_{SR}}(\D)$ = Rk(supp($\D$)) for some $\D\subseteq \C$ with $\dim \D= i$,
    and $d_{SR,i}(\C)=\Rk(\L)$ for some $\L\in \P(\K^{\n})$. Let $\L_{0}=\text{supp}(\D)$, so that ${\text{wt}_{SR}}(\D)=\Rk(\L_{0})$. By definition,
    \[
    \C\cap \Nu_{\L_{0}}=\{\x\in \C : \text{supp}(\x) \subseteq \L_{0}\}.
    \]
    So we have $\D \subseteq \C \cap \Nu_{\L_{0}}$ and thus $\text{dim}_{\F}(\C \cap \Nu_{\L_{0}})\geq i$. Therefore by definition, $d_{SR,i}(\C) \leq \Rk(\L_{0})={\text{wt}_{SR}}(\D)$, i.e., 
    \begin{align}\label{eq::104}
        d_{SR,i}(\C) \leq \min\{{\text{wt}_{SR}}(\D) \colon \D \subseteq \C \text{ and } \dim_{\F}\D = i \}.
    \end{align}
    Conversely, since $d_{SR,i}(\C)=\Rk(\L)$, we have $\text{dim}_{\F}(\C \cap \Nu_{\L})\geq i$. Let $\D\subseteq \C \cap \Nu_{\L}$. Then $\text{supp}(\D)\subseteq \L$, implying that ${\text{wt}_{SR}}(\D)\leq \Rk(\L)$. Hence by definition,
    \begin{align}\label{eq::105}
      \min \{{\text{wt}_{SR}}(\D) \colon \D \subseteq \C \text{ and } \dim_{\F}\D = i  \}\leq {\text{wt}_{SR}}(\D)\leq \Rk(\L)=d_{SR,i}(\C).
    \end{align}
    Combining Equations \eqref{eq::104} and \eqref{eq::105} we get 
    $d_{SR,i}(\C)=\min \{{\text{wt}_{SR}}(\D) \colon \D \subseteq \C \text{ and } \dim_{\F}\D = i  \}$.
    
\end{proof}
Since $d_{SR,1}(\C)$ is the minimum distance of the sum-rank metric code $\C$, we simply use $d_{SR}(\C)$ or $d$ instead of $d_{SR,1}(\C)$, when there is no ambiguity and we describe the code as  $[n,k,d;n_1,\dots,n_{\ell}]$-sum-rank metric code.

We have seen in the beginning of this subsection that for the cases $\n=(1,\ldots,1)$ and $l=1$, the sum-rank metric $d_{SR}$ recovers the Hamming metric and the rank metric, respectively. It is equivalent to saying that the notion of sum-rank weight $wt_{SR}$ is a generalization of the notions of Hamming weight $wt_H$ and the rank weight $wt_{R}$ of a codeword, where
\begin{equation} wt_H(\x) = d_H(\x, \0) \quad \text{ and } \quad wt_R(\x) = d_R(\x,\0).\end{equation}
Thus from Lemma \ref{genweight}, it is clear that for the particular cases $\n=(1,\ldots,1)$ and $l=1$, we get back the generalized Hamming weights \cite{Wei91} and the generalized rank weights \cite{JP2017,KMU15,KUM12,OS12}, respectively.

The following results generalize the monotonicity and Wei duality theorems for Hamming and rank metric codes.
\begin{pro}[Monotonicity]\label{mon}
For an $[n,k,d;n_1,\dots,n_{\ell}]$-sum-rank metric code $\C$ with $k> 0$, the generalized weights of $\C$ satisfy $1\leq d_1<d_2< \cdots < d_k\leq n.$
\end{pro}
\begin{proof}
For a proof, see \cite[Lemma 4]{MP2019}.
\end{proof}
\begin{rem}
For Hamming metric codes, a proof can be found in \cite[Theorem 1]{Wei91}. For rank metric codes, one can see (\cite[Lemma 9]{KMU15}).
\end{rem}
For Hamming metric codes, Wei proved in \cite{Wei91} that the set of generalized weights of a linear code determines the set of generalized weights of its dual. Later Wei-type duality theorems have been proved for codes with rank metric and sum-rank metric. Here we state the theorem for sum-rank metric codes which covers the particular cases of the Hamming and the rank-metric codes. 

\begin{thm}[Wei duality]\label{thm@5}
Let $\C$ be an $[n,k,d;n_1,\dots,n_{\ell}]$-sum-rank metric code. Let $(d_1,d_2,\ldots,d_k)$ and $(d_1^\perp,d_2^\perp,\ldots,d_{n-k}^\perp)$ denote the weight hierarchies of $\C$ and $\C^\perp$, respectively. Then \[\{d_t:1\leq t\leq k\}=\{1,2,\ldots,n\} \setminus \{n+1-d_r^\perp:1\leq r\leq n-k\}.\]
\end{thm}
\begin{proof}
For a proof, see \cite[Theorem 2]{MP2019}.
\end{proof}
\begin{rem}
For a proof in the case of Hamming metric, one can see \cite[Theorem 3]{Wei91} and for rank metric codes, see, e.g., \cite[Thm. I.3]{BMS19},\cite{DK15}.
\end{rem}

\begin{thm}[Generalized Singleton bound,~\cite{MP2019}]
For an $[n,k,d;n_1,\dots,n_{\ell}]$-sum-rank metric code $\C$, the generalized sum-rank weights satisfy
\[
d_{SR,i}(\C) \leq n-k+i.
\]
\end{thm}

\begin{defn}\label{def@13}
Given a $k$-dimensional linear code $\C \subseteq \F^n$, its  \emph{maximum sum-rank distance} (MSRD) rank is defined as the minimum integer $i\in\{1,\ldots,k\}$ such that $d_{SR,i}(\C)=n-k+i$, if such an $i$ exists. In such a case, we say that $\C$ is an $i$-MSRD code.  $1$-MSRD codes are simply called MSRD codes.
\end{defn}

Examples of MSRD codes can be found, e.g., in \cite{MP2019, PR19}.
The bound on $d_1$ in the previous theorem is called the \emph{Singleton} bound \cite{Del78,Gab85,Sin64}. Hamming (resp. rank) metric codes attaining the Singleton bound are called maximum distance separable (MDS) (resp. maximum rank distance (MRD)) codes. For more about the properties of these classes of codes, one can look at \cite{Del78,Gab85,HM17,MS78,NHRR18}. If $i$ is the smallest integer such that $d_i = n -k+i$, then the code is called $i$-MDS (resp. $i$-MRD) in case of Hamming (resp. rank) metric codes.

The following result gives a characterization of MDS codes.
\begin{thm}[{\cite[Chap. 11, Corollary 3]{MS78}}]\label{6} 
Let $\G$ be a generator matrix of an $[n,k]$-linear code $\C$ over $\Fq$. Then $\C$ is MDS if and only if every set of $k$ distinct columns of $\G$ are linearly independent.
\end{thm}

For rank metric codes, the characterization is stated as follows. 

\begin{thm}[\cite{Gab85}]\label{7}
Let $\G$ be a generator matrix of an $[n,k]$-rank metric code $\C$ over the extension $\Fqm/\Fq$. Then $\C$ is MRD if and only if the product of matrices $\G\M$ is invertible for every $n\times k$ matrix $\M$ of rank $k$ over $\Fq$.
\end{thm}

As a consequence of these statements it can be easily seen that $\C$ is MDS (resp. MRD) if and only if $\C^\perp$ is MDS (resp. MRD).
 
 Here we give a new characterization of MSRD codes which combines Theorems \ref{6} and \ref{7}.

\begin{thm}\label{thm::24}
Let $\C$ be an $[n,k,d;n_1,\dots,n_{\ell}]$-sum-rank metric code  over $(\F;K_1,\dots,K_{\ell})$ with generator matrix $\G$. Then $\C$ is MSRD if and only if $\G \AA$ is invertible for every $n\times k$ matrix $\AA$ of rank $k$ such that $\AA=\diag(\AA_1,\dots,\AA_{\ell})$ where each $\AA_i$ is an $n_i\times k_i$ matrix over $K_i$ , $k_i\leq n_i$, and $\sum_{i=1}^{\ell} k_i = k$.
\end{thm}

\begin{proof}
Assume that $\C$ is MSRD. Let $\AA=\diag(\AA_1,\dots,\AA_{\ell})$ with the properties stated in the theorem. Consider the $\F$-linear map
$$
\F^k  \rightarrow \F^k \quad \text{ defined by } \quad
\x  \mapsto \x \G\AA.$$

It is enough to show that this map is injective. Assume that $\x \G\AA=0$. Thus $(\x \G)\AA = 0$. Notice that $\x \G$ is a codeword of $\C$ and write it as $\x \G=(\c^{(1)},\dots,\c^{(\ell)})$. Hence $\c^{(i)} \AA_i = 0$ for $i=1,\dots,\ell$. This implies that $\text{rank}_{K_i}(\Gamma_i(\c^{(i)}))\leq n_i-k_i$. Taking the sum we get $\sr (\x \G)\leq n-k$. Since $\x \G$ is a codeword of $\C$ and $\C$ has minimum distance $n-k+1$, we see that $\x \G$ must be the zero codeword and thus $\x$ itself is the zero vector. Therefore our map is injective which implies that $\G\AA$ is invertible.
	
Conversely, suppose that $\G\AA$ is invertible for every $n\times k$ matrix $\AA$ with the properties in the theorem. Let $\c$ be a codeword of $\C$. If $\c\in \C$ has sum rank smaller than $n-k+1$ then assume that $\c = (\c^{(1)},\dots,\c^{(\ell)})$ such that $\text{rank}_{K_i}(\Gamma_i(\c^{(i)})) = r_i$ and $\sum_{i-1}^{\ell} r_i\leq n-k$. 

For each $i$, we can find an $n_i\times (n_i-r_i)$ matrix $\AA_i$ of rank $n_i-r_i$ over $K_i$ such that $\c^{(i)} \AA_i = 0\in \F^{n_i-r_i}$. Let $\AA=\diag(\AA_1,\dots,\AA_{\ell})$. Thus $\c \AA=0$ where $\AA$ has $u=\sum_{i=1}^l (n_i-r_i)$ columns and $\text{rank}(\AA)=u$. Since $u \geq k$, we can choose a submatrix $\AA'$ of $\AA$ of rank $k$ such that  $\c \AA'=0$.
As a codeword of $\C$, $\c$ can be written as $\c = \x \G$ and therefore $\x \G\AA'=0$ and therefore we have a contradiction since $\G\AA'$ is invertible.

%just need to select $k$ columns from this matrix to construct a matrix $\AA'$ of rank $k$ such that $\c \AA'=0$. 
\end{proof}

\subsection{($q$-)Matroids }\label{Sec:2.2}

We now recall the notion of matroids and its $q$-analogue, called  $q$-matroids. These are combinatorial objects which can be seen as generalizations of linear codes embedded with Hamming metric or rank metric. We shall see in the next section how these are generalized into a common notion of sum-matroids.

\begin{defn}[Matroids]
A matroid $M=(E,\rho)$ is a pair such that $E$ is a finite set and $\rho$ is a function $2^E \rightarrow \N_0$, called the {\em rank} function, satisfying the following properties for all $A, B \subseteq E$:
\begin{enumerate}[(r1)]
\item \label{r1}   $\rho(\emptyset)=0$, 
\item \label{r2} if $A\subseteq B$, then $\rho(A) \leq \rho(B)$,
\item \label{r3} $\rho(A \cap B) + \rho(A \cup B) \leq \rho(A) + \rho(B)$.
\end{enumerate}
The rank of the matroid $M$, denoted as $\rho(M)$, is defined to be $\rho(E)$.
\end{defn}

For the other cryptomorphic definitions and more detailed treatment of matroids one can look at \cite{Ox06,Whi86}.

\begin{exa}
Let $E=[n]:=\{1,\dots, n\}$. 
For $k\leq n$, let $\H$ be a $k\times n$ matrix over a finite field $\Fq$. For a subset $A\subseteq E$, let $\H_A$ be the submatrix of $\H$ with the columns indexed by $A$.
Let $\rho$ be the function defined by 
\begin{align*}
    \rho:2^E & \longrightarrow \N_0 \\
    A &\longmapsto \text{rank}(\H_A).
\end{align*}
Then $(E,\rho)$ defines a matroid denoted by $M_\H$. 
\end{exa}

A matroid which can be defined as in the previous example is called a {\em representable matroid}. There are examples of matroids which are not representable for any finite field $\Fq$ (\cite[Chap. 3]{Ox06}). Representable matroids can be thought as a generalization of linear codes as we see in the following definition.
\begin{defn}
    Let $\C$ be an $[n,k]$-linear code over $\Fq$. Let $\H$ be a parity check matrix of $\C$, the matroid associated to $\C$ is the representable matroid defined on $E=[n]$ by the matrix $\H$. 
\end{defn}

Following this definition, it is natural to generalize some invariants from linear codes to matroids. The next definition is a generalization of the generalized weights of linear codes to matroids \cite{JV13}.

\begin{defn}[Generalized weights of matroids]\label{defn:22}
Let $M=(E,\rho)$ be a matroid. For $1\leq i \leq n-\rho(E)$, the $i$-th generalized weight of $M$ is defined as 
\[
d_i = \min_{\sigma\in \P(E)}\{ |\sigma|: |\sigma|-\rho(\sigma) = i\}.
\]
\end{defn}

\begin{defn}[Dual matroids \cite{Ox06}]%\cite{BJMS12}
Let $M=(E,\rho)$ be a matroid. The dual matroid of $M$ is the matroid $M^* = (E,\rho^*)$, where for $U \subseteq E$,
\[
\rho^*(U) = |U| + \rho(E \backslash U) - \rho(M).
\]
\end{defn}

Next we describe the notion of $q$-matroids which are the $q$-analogue of the matroids as introduced in \cite{JP2016}. These are $q$-analogues in the sense that the finite sets for matroids are replaced with finite dimensional vector spaces over a finite field $\Fq$ and cardinality of a set is replaced by dimension of a space. Let $V=\Fq^n$ be an $n$-dimensional vector space over $\Fq$ and with abuse of notation, let $\P(V)$ be the set of all subspaces of $V$.

\begin{defn}[$q$-Matroids]
A $q$-matroid $M=(V,\rho)$ is a pair such that $V$ is a vector space of finite dimension of $\Fq$ and $\rho$ is a function $\P(V)\rightarrow \N_0$ called the {\em rank function} satisfying the following properties for all $A, B \subseteq E$ and $a \in E$:
\begin{enumerate}[($\rho$1)]
\item \label{rho1}   $\rho(\{\0\})=0$, 
\item \label{rho2} if $A\subseteq B$, then $\rho(A) \leq \rho(B)$,
\item \label{rho3} $\rho(A \cap B) + \rho(A + B) \leq \rho(A) + \rho(B)$.
\end{enumerate}
The rank of the $q$-matroid $M$, denoted as $\rho(M)$, is defined to be $\rho(V)$.
\end{defn}

Similar to matroids, various cryptomorphic defitions of $q$-matroids can be found in \cite{BCR21,JP2016}.
The following classes of $q$-matroids are given by Jurrius and Pellikaan in \cite{JP2016}.

\begin{exa}
Let $\H$ be a $k\times n$ matrix over a finite field $\Fqm$ and let $\D$ be its rowspace. Let $U\subseteq \Fq^n$ be an $\Fq$-subspace of dimension $t$ and assume that $\Y$ is a generator matrix of $U$.
 We define a map $\pi_U\colon \Fqm^n\rightarrow \Fq^t$ such that $\pi_U(\x) = \x\Y^T$. Taking $\D_U = \pi_U(\D)$ we define rank function $\rho$ by $\rho(U) = \dim_{\Fq} (\D_U)$. In fact we have $\rho(U) = \text{rank}(\H \Y^T)$ and it has been shown in \cite[Theorem 4.8]{JP2016} that $\rho$ indeed defines a rank function of a $q$-matroid $M_{\D} = (\Fq^n,\rho)$ of rank $k$ associated to the matrix $\H$.
\end{exa}

A $q$-matroid which can be defined as in the previous example is also called a {\em representable $q$-matroid} over $\Fq$. There are examples of non-representable $q$-matroids for any finite field $\Fq$  \cite{BCR21,CJ21}. The representable $q$-matroids lead to the notion of $q$-matroids associated to a rank metric code.
\begin{defn}
    Let $\C$ be an $[n,k]$-linear code over $\Fqm/\Fq$. Let $\H$ be a parity check matrix of $\C$, the $q$-matroid associated to $\C$ is the representable $q$-matroid defined on $\Fq^n$ by the matrix $\H$. 
\end{defn}

\begin{rem}
Our definition is slightly different from the original definition in \cite{JP2016}. In that paper, the matroid associated to a rank metric code is defined as the representable matroid associated to its generator matrix. In order to show the generalization to sum-matroids and also the analogy to classical matroids, it is desirable to work with the parity check matrix.
\end{rem}

As the notion of $q$-matroids is a generalization of the notion of rank metric codes, we can also extend the notion of generalized weights to $q$-matroids. Here we state the definition from \cite{GJ19}. 
 In \cite{GJ19}, the generalized weights of \linebreak $(q,m)$-polymatroids, which corresponds to (matrix) rank metric codes, are defined. But any vector rank metric code can be seen also as a (matrix) rank metric code. Hence the generalized weights of $q$-matroids are just particular cases those for $(q,m)$-polymatroids. In what follows, whenever we mention the results from \cite{GJ19}, it will be understood that these are for the special class of Delsarte rank metric codes which arise from vector rank metric codes or $(q,m)$-polymatroids with $m=1$. The case $m=1$ is further studied in \cite{GJLR19}.

\begin{defn}[Generalized weights of $q$-matroids]\label{defn:28}
Let $V$ be an $n$-dimensional vector space over $\Fq$ and let $M=(V,\rho)$ be a $q$-matroid. The $i$-th generalized weights of $M$ is defined as
\[
d_i^R(M) = \min\{ \dim_{\Fq} U\colon U\in \P(V) \text{ and } \dim_{\Fq} U - \rho(U)= i \},
\]
for $i=1, \dots, n - \rho(V)$.
\end{defn}

\begin{defn}[Dual $q$-matroids \cite{JP2016}]
Let $M = (V,\rho)$ be a $q$-matroid. The dual $q$-matroid $M^*$ of $M$ is the $q$-matroid $M^* = (V,\rho^*)$,  where for $U\in \P(V)$,
\[
\rho^*(U) = \dim_{\Fq} U + \rho(U^\perp) - \rho(M).
\]
\end{defn}

We finish this section with a generalization of two properties of the generalized weights of linear codes. The first property is about the monotonicity which generalizes the Proposition \ref{mon}. The second property generalizes the Wei duality for linear codes in Theorem \ref{thm@5} to ($q$-)matroids. For simplicity, we combine the statements for both matroids and $q$-matroids.

\begin{pro}[Monotonicity \cite{BJMS12}]\label{pro:mono} Let $M=([n], \rho)$ (resp. $M=(\Fq^n, \rho)$) be a matroid (resp. $q$-matroid) of rank $k$. For $1 \leq i\leq n-k$, if $d_i$'s are the generalized weights of $M$, then $d_1 < d_2 < \cdots < d_{n-k}$. 
\end{pro}

The Wei-type duality theorems for the Hamming and the rank-metric codes have been generalized for matroids and $q$-matroids, respectively.
 \begin{thm}[Wei-type duality \cite{GJ19}]\label{Wei:matroids}
	Let $M=([n], \rho)$ (resp. $M=(\Fq^n, \rho)$) be a matroid (resp. $q$-matroid) of rank $k$ and $M^*=([n], \rho^*)$ (resp. $M^*=(\Fq^n, \rho^*)$) be its dual. For $i =1, \ldots, n-k$ and for $j =1, \ldots, k$, if we let $d_i$ to be the $i$-th generalized weight of $M$ and  $d_j^{\perp}$ the $j$-th generalized weight for $M^*$, then we have
	\[
	\{1, 2, . . . , n\} = \{d^{\perp}_1, \ldots , d^{\perp}_k \}\cup\\
	\{n + 1 - d_1, \ldots, n+1-d_{n-k} \},
	\]
	where the union is disjoint. In particular, the generalized weights of a matroid (resp. $q$-matroid) $M$ uniquely
	determine those of $M^*$. 
\end{thm}

\section{The notion of sum-matroid}\label{Sec:3}

In this section, we introduce the notion of \emph{sum-matroid} using rank function and define its dual. We shall see how this notion generalizes both the notions of matroid and $q$-matroid. 

We follow the notations declared in Section \ref{Sec:2}. Recall that we use $\n$ and $\K$ to denote the $\ell$-tuples $(n_1,\ldots,n_{\ell})$ and $(K_1,\ldots,K_{\ell})$, respectively and $\K^{\n}$ to denote the $\ell$-tuple of vector spaces $(K_1^{n_1},\ldots,K_{\ell}^{n_\ell})$. Also, $\P(\K^{\n})$ is the cartesian product of lattices of subspaces as given in Definition \ref{def@7}. The \emph{sum-matroids} on $\K^{\n}$ are defined as follows.

\begin{defn}\label{def@14}
A \emph{sum-matroid} $M$ is a pair $(\K^{\n},\rho)$, where $\rho$ is a non-negative integer-valued function defined on $\P(\K^{\n})$, such that for any $\L, \L^{\prime} \in \P(\K^{\n})$, 
	\begin{enumerate}[(R1)]
		\item\label{R1} $0 \leq \rho(\L) \leq \text{Rk}(\L)$,
		\item\label{R2} if $\L \subseteq \L^{\prime}$, then $\rho(\L) \leq \rho(\L^{\prime})$,
		\item\label{R3} $\rho(\L + \L^{\prime})+ \rho(\L \cap \L^{\prime}) \leq \rho(\L)+\rho(\L^{\prime})$.
	\end{enumerate}
\end{defn}
 \begin{defn}
	Let $M=(\K^{\n},\rho)$ be a sum-matroid. An element $\L$ of $\P(\K^{\n})$ is called \emph{independent} if $\rho(\L)=\text{Rk}(\L)$, otherwise $\L$ is called \emph{dependent}. Moreover, if $\L$ is independent and $\rho(\L)=\rho(\K^{\n})$, we call $\L$ a \emph{basis} of $M$. The \emph{rank} of $M$, denoted by $\rho(M)$,  is defined to be $\rho(\K^{\n})$.
\end{defn}

The function $\rho$ in Definition \ref{def@14} is called \emph{rank function}. Note that the rank function is defined by three properties, i.e., bound on the rank function (R\ref{R1}), monotonicity (R\ref{R2}) and semimodularity (R\ref{R3}). 
Indeed, this resembles the definitions of rank function for matroids and $q$-matroids in Section \ref{Sec:2.2}.
For the rank function of matroids is defined on  $2^E$, where $E$ is a finite set, say, $[n]:=\{1, \ldots, n\}$ and it satisfies the same properties (R\ref{R1}), (R\ref{R2}) and (R\ref{R3}) with the only exception that in (R\ref{R1}), the rank  $\rho(A)$ of an element $A \subseteq E$ is bounded by $|A|$. 

On the other hand, the rank function for $q$-matroids is defined on $\P(\Fq^n)$, the lattice of subspaces of $\Fq^n$. In this case, the only difference is that  in (R\ref{R1}), the rank  $\rho(A)$ of an element $A \subseteq \Fq^n$ is bounded by $\dim(A)$. This explains the ``$q$" of $q$-matroid as we replace the finite sets and cardinality of a set by their $q$-analogues, i.e., finite dimensional vector spaces and the dimension of a space, respectively.

 Note that, for $\ell=1$, $\P(\K^{\n})$ becomes $\P(K_1^{n_1})$ which gives the domain of rank functions of $q$-matroids. On the other hand, taking $\n = (1,1, \ldots, 1)$, i.e., $n_i=1$ for $1\leq i\leq \ell$, $\P(\K^{\n})$ becomes $2^{[l]}$ which recovers the domain of rank functions of matroids.

\begin{defn}\label{def@16}
	Let $M = (\K^{\n},\,\rho)$ be a sum-matroid. The dual sum-matroid $M^*$ of $M$ is the sum-matroid $(\K^{\n},\, \rho^*)$, where for any $\L \in  \P(\K^{\n})$,
	\[
	\rho^*(\L) := \rho(\L^\perp) + \rm{ Rk } (\L) - \rho(\K^{\n}).
	\]
\end{defn}

We show next that $\rho^*$ is indeed a rank function on $\P(\K^{\n})$.

\begin{thm}
If $\rho$ is a rank function defined on $\P(\K^{\n})$, then $\rho^*$ is also a rank function on $\P(\K^{\n})$.
\end{thm}
\begin{proof}
    First we show that $\rho^*$ satisfies (R\ref{R2}). We let $\L\subseteq \L^{\prime}$. This implies $\L'^{\perp}\subseteq \L^{\perp}$, i.e., there exists $\L_0 \in \P(\K^{\n})$ such that $\L'^{\perp}\oplus \L_0=\L^{\perp}$ and $\rm{Rk}(\L_0)=\rm{Rk}(\L^{\perp})-\rm{Rk}(\L'^{\perp})$. Then
    \[
    \rho(\L^{\perp})=\rho(\L'^{\perp}+\L_0)\leq \rho(\L'^{\perp})+\rm{Rk}(\L_0)=\rho(\L'^{\perp})+\rm{Rk}(\L^{\perp})-\rm{Rk}(\L'^{\perp}), 
    \]
     where the inequality comes from property (R\ref{R3}) of $\rho$. Therefore,
    \begin{align}\label{eq::101}
        \rho(\L^{\perp})-\rm{Rk}(\L^{\perp})\leq \rho(\L'^{\perp})-\rm{Rk}(\L'^{\perp})
    \end{align}
    
    Since $\rm{Rk}(\L)+\rm{Rk}(\L^{\perp})=n$, then Equation \eqref{eq::101} becomes 
    \[
        \rho(\L^{\perp})+\rm{Rk}(\L)\leq \rho(\L'^{\perp})+\rm{Rk}(\L'), 
    \]
    and from definition of $\rho^*$, it follows that $\rho^*(\L)\leq \rho^*(\L')$, which proves (R\ref{R2}). 
    
    Now we show that $\rho^*$ satisfies (R\ref{R1}). Note that by definition, $\rho^*(\0)=0$ and $\0 \subseteq \L$ for any$\L \in \P(\K^{\n})$. Then (R\ref{R2}) of $\rho^*$ implies
    \[
    0= \rho^*(\0)\leq \rho^*(\L).
    \]
    
    Also, by property (R\ref{R2}) of $\rho$, given $\L \in \P(\K^{\n})$, we have $\rho(\L^{\perp}) \leq \rho(\K^{\n})$.
    Hence we have 
    \[
    \rm{Rk}(\L)+\rho(\L^{\perp})-\rho(\K^{\n})\leq \rm{Rk}(\L),
    \]
    i.e., $\rho^*(\L)\leq \rm{Rk}(\L)$ and that proves (R\ref{R1}) for $\rho^*$.

    Finally to prove (R\ref{R3}), let $\L,\L' \in \P(\K^{\n})$. Then 
    \begin{align*}
	 \rho^*(\L+\L')+\rho^*(\L\cap \L') & =\rho(\left(\L+\L'\right)^\perp) + \text{ Rk }(\L+\L') - 2\rho(\K^{\n}) \\
	&\qquad \quad + \rho(\left(\L\cap \L'\right)^\perp) + \text{ Rk }(\L\cap \L') \\
	& =\rho(\left(\L+\L'\right)^\perp) + \text{ Rk }(\L)+\text{ Rk }(\L') - 2\rho(\K^{\n}) \\
	&\qquad \quad + \rho(\left(\L\cap \L'\right)^\perp).
	\end{align*}
	But $(\L+\L')^\perp =\L^\perp \cap \L'^\perp$ and $(\L\cap\L')^\perp =\L^\perp + \L'^\perp$. Hence
	\begin{align*}
	& \rho^*(\L+\L')+\rho^*(\L\cap \L') \\
	&\qquad =\rho(\L^\perp \cap \L'^\perp) + \text{ Rk }(\L)+\text{ Rk }(\L') - 2\rho(\K^{\n}) \\
	&\qquad \quad + \rho(\L^\perp + \L'^\perp) \\
	&\qquad \leq \rho(\L^\perp)+\rho(\L'^\perp) + \text{ Rk }(\L)+\text{ Rk }(\L') - 2\rho(\K^{\n}) \\
	&\qquad = \rho^*(\L) + \rho^*(\L').
	\end{align*}
	And this concludes the proof.
\end{proof}

Note that the definition of dual sum-matroid in Definition \ref{def@16} naturally  generalizes the respective definitions for matroids and $q$-matroids as mentioned in Section \ref{Sec:3}.

% The dual matroid of $M =(E,\rho)$ is the matroid $M^* = (E,\rho^*)$, where 
% \[
% \rho^*(U) = |U| + \rho(E \backslash U) - \rho(M),
% \]
% for $U \subseteq E$.
% The above definition can be found in \cite{BJMS12}. Next is the definition of dual of $q$-matroids from \cite{JP2016}. 

% The dual $q$-matroid $M^*$ of  is the $q$-matroid $M^* = (E,\rho^*)$ where 
% \[
% \rho^*(U) = \dim_{\Fq} U + \rho(U^\perp) - \rho(M),
% \]
% for $U$ a subspace of $E$. 

\section{Sum-matroids from sum-rank metric codes}\label{Sec:4}
We shall now associate a sum-matroid to a sum-rank metric code $\C \subseteq \F^{n}$ and show how duality is preserved in this association. To do that, we first define some subspaces of the dual sum-rank metric code $\C^\perp$.
%One of the motivations of studying sum-matroids is that we can associate a sum-matroid to a sum-rank metric code $\C$. For such sum-matroids, we define its generalized weights which correspond to the genralized weights of the associated linear code. This will allow us to use combinatorial methods to prove properties of generalized weights of linear codes. 
%Notably, in the case of matroids, one can associate simplicial complexes which define topological spaces and Stanley-Reisner rings. And it turns out that the generalized weights of the codes are related to the singular homology degrees of the topological spaces and to the Betti numbers of the Stanley-Reisner ring \cite{JV13}.
%To construct a sum-matroid from a given sum-rank metric code we start by defining some subspaces of the dual sum-rank metric code $\C^\perp \subseteq \F^{n}$.

 \begin{defn}\label{defn:36}
	Let $\C \subseteq \F^{n}$ be an $\F$-linear subspace and  $\L \in \P(\K^{\n})$. We define:
\[
\C(\L):=\{\c\in \C^\perp :\rm{supp}(\c) \subseteq \L^\perp\}.
\]
\end{defn}

Note that from Definitions \ref{def@11} and \ref{defn:36} it follows that

\begin{equation}\label{rem}
    \C(\L) = \Nu_{\L^{\perp}} \cap \C^\perp.
\end{equation} 
Next we give a characterization of $\C(\L)$ that will be useful in the sequel. 

\begin{lem}\label{23}
	Let $\C$ be an $\F$-linear code of length $n$ and  $\L \in \P(\K^{\n})$. Then $\c \in \C(\L)$ if and only if $\c \cdot \y=0$ for all  $\y \in \L$, where $\c \cdot \y$ is the usual dot product in $\F^n$. Furthermore, $\C(\L)$ is an $\F$-linear subspace of $\C^\perp$.
\end{lem}

\begin{proof}
	Let $\c=(\c^{(1)}, \ldots, \c^{(\ell)}) \in \C^\perp$ and  $\text{supp}(\c)=(E_1, \ldots , E_{\ell})$. Then  
	\begin{align*}
	\c \in \C(\L)  &\iff \text{supp}(\c) \in \L^{\perp}\\
	&\iff E_i \subseteq \L_i^{\perp} \; for \; all \; i=1,\ldots,\ell.\\
	&\iff \text{supp}(\c^{(i)}) \subseteq \L_i^{\perp} \; for \; all \; i=1,\ldots,\ell.\\
	&\iff \c^{(i)} \cdot \y^{(i)}=0 \; for \; all \; \y^{(i)} \in \L_{i}\\
	&\iff \c \cdot \y =0 \; for \, all \; \y \in \L.
	\end{align*}
	%For the last equivalence, $\c^{(i)} \cdot \y^{(i)}=0 \; for \; all \; \y^{(i)} \in \L_{i}\Rightarrow \c \cdot \y =0 \; for \, all \; \y \in \L$ is trivial. The converse comes from the fact that if $\c \cdot \y =0 \; for \, all \; \y \in \L$ then in particular for every $i$, $\c \cdot \y =0 \; for \, all \; \y \in \L$ for $\y=(\0,\dots,\0,\y^{(i)},\0,\dots,\0)$.
	
	The converse in the very last equivalence, i.e., $\c \cdot \y =0 \text{ for all } \y \in \L \Longrightarrow \c^{(i)} \cdot \y^{(i)}=0 \text{ for all } \y^{(i)} \in \L_{i}$ follows from the particular choices of $\y=(\0,\dots,\0,\y^{(i)},\0,\dots,\0) \in \L$.
	
	Finally, that $\C(\L)$ is an $\F$-linear subspace of $\F^{n}$ follows from the observation that $\C(\L)=\Nu_{\L^{\perp}} \cap \C^\perp$, which is intersection of two $\F$-linear subspaces.
\end{proof}

From the above lemma we can alternatively write 
\begin{equation}\label{eq::2}
\C(\L)=\{\c \in \C^\perp:  \c \cdot \y =0  \text{ for } \text{all }  \y \in \L \}.
\end{equation}

\begin{defn}\label{def@38}
Let $\L =(\L_1,\ldots,\L_{\ell})$ be an element of $\P(\K^{\n})$ of rank $N$ with $\dim_{K_i}(\L_i)=N_i$ for $1\leq i \leq \ell$, where $N= N_1 + \cdots + N_{\ell}$. Let $\AA_i \in M_{ N_i \times n_i}(K_i)$ be a generator matrix of $\L_i$, for $1\leq i \leq \ell$ and $\AA=\diag(\AA_1,\ldots,\AA_{\ell})$, a block diagonal matrix. Define the $\F$-linear map 
\begin{align}\label{piL}
  \nonumber  \Pi_{\L} \colon\F^n &\longrightarrow \F^N\\ \x &\longmapsto \x \AA^T.
\end{align}
\end{defn}

Next we define another subspace of $\F^n$ related to $\C$ which will be used to define the rank function of a sum-matroid. 
\begin{defn}\label{def@18}
Let $\C \subseteq \F^{n}$ be an $\F$-linear code and $\C^{\perp}$ be its dual. Then for any $\L \in \P(\K^{\n})$, we define 
		   \begin{align}\label{CL}
		   \C_{\L}=\Pi_{\L}(\C^\perp).
		  \end{align}
\end{defn}		   
Note that, $\text{ker}(\Pi_{\L}) \cap \C^\perp=\{\c \in \C^\perp:  \c \cdot \y =0 \; \text{for} \; \text{all} \; \y \in \L \}$, which is equal to $C(\L)$ from equation \eqref{eq::2}. So from the definition of $\C_{\L}$ in Definition ~\ref{def@18}, we obtain a short exact sequence of vector spaces over the field $\F$,
	\begin{equation}\label{eq::3}
	0 \longrightarrow \C(\L) \longrightarrow \C^\perp \overset{\Pi_{\L}|_{\C^\perp}}{\longrightarrow} \C_{\L} \longrightarrow 0.
	\end{equation}

\begin{defn}\label{def@19}
	Let $\C$ be an $[n,k,d;n_1,\dots,n_{\ell}]$-sum-rank metric code over $(\F;K_1,\dots,K_{\ell})$ and $\L \in \P(\K^{\n})$. For any $\L \in \P(\K^{\n})$ we define 
	\[
	\mu_{\C}(\L):=\dim_{\F}(\C(\L)) \text{ and } \rho_{\C}(\L):= \dim_{\F}(\C_{\L}).
	\]
	\end{defn}

From the short exact sequence in \eqref{eq::3}, we obtain the following lemma.

\begin{lem}\label{27}
	Let $\C$ be an $[n,k,d;n_1,\dots,n_{\ell}]$-sum-rank metric code over 
	\linebreak $(\F;K_1,\dots,K_{\ell})$ and $\L \in \P(\K^{\n})$. Then  $\mu_{\C}(\L) +  \rho_{\C}(\L)=\dim \C^{\perp} =n-k$.
\end{lem}

Now we come to our main result of this section. We show that a \emph{sum-matroid} can be associated to a \emph{sum-rank metric} code $\C$ by proving that $\rho_{\C}$ gives a rank function on $\P(\K^{\n})$.

\begin{thm}\label{54}
	Let $\C$ be an $[n,k,d;n_1,\dots,n_{\ell}]$-sum-rank metric code over $(\F;K_1,\dots,K_{\ell})$ and $\rho_{\C}$ be the function as given in Definition \ref{def@19}. Then $(\K^{\n},\rho_{\C})$ defines a sum-matroid.
\end{thm}

\begin{proof}
	First of all it is clear that $\rho_{\C}$ is a non-negative integer valued function defined on $\P(\K^{\n})$. We need to show that $\rho_{\C}$ satisfies the properties (R\ref{R1}), (R\ref{R2}), (R\ref{R3}) of Definition \ref{def@14}. Let $\L,\L^{\prime}\in \P(\K^{\n})$. Lemma \ref{27} implies that $\rho_{\C}(\L)=n-k-\mu_{\C}(\L)$. 
	\begin{enumerate}[(R1)]
		\item $0 \leq \rho_{\C}(\L) \leq \text{Rk}(\L)$:
		
	 By definition \ref{def@19}, $\rho_{\C}(\L)= \text{dim}_{\F}(\C_{\L})$ and from Definitions \ref{def@38} and \ref{def@18}, we know that $\C_{\L}$ is a subspace of $\F^{N}$ if $\text{Rk}(\L)=N$. This proves (R\ref{R1}).
	
		\item If $\L \subseteq \L^{\prime}$, then $\rho_{\C}(\L) \leq \rho_{\C}(\L^{\prime})$:
	
	Let $\L \subseteq \L^{\prime}$ and $\c \in \C(\L^{\prime})$. Then $\L \subseteq \L^{\prime} \subseteq \text{supp}(\c)^{\perp}$. So $\c \in \C(\L)$. Hence $\C(\L^{\prime}) \subseteq \C(\L)$ and $\mu_{\C}(\L^{\prime}) \leq \mu_{\C}(\L)$. Therefore, by Lemma \ref{27}, $\rho_{\C}(\L) \leq \rho_{\C}(\L^{\prime})$.

	\item  $\rho_\C(\L + \L^{\prime})+ \rho_\C(\L \cap \L^{\prime}) \leq \rho_{\C}(\L)+\rho_{\C}(\L')$:
	
	First note that $\C(\L) \cap \C(\L')=\C(\L + \L^{\prime})$, since
	\begin{align*}
	\c \in  \C(\L) \cap \C(\L') &\iff \c \in \C(\L) \;\text{and} \; \c \in C(\L')\\
	&\iff \L \subseteq \text{supp}(\c)^{\perp} \; \text{and} \; \L' \subseteq \text{supp}(\c)^{\perp}\\
	&\iff \L + \L^{\prime} \subseteq \text{supp}(\c)^{\perp}\\
	&\iff \c \in \C(\L +\L').
	\end{align*}
	
	Now $\L \cap \L^{\prime} \subseteq \L$ implies that $\C(\L) \subseteq \C(\L \cap \L^{\prime})$ and similarly, $\C(\L^{\prime}) \subseteq \C(\L \cap \L^{\prime})$. Since $\C(\L \cap \L^{\prime})$ is a subspace of $\C^\perp$, it follows that  $\C(\L)+\C(\L^{\prime}) \subseteq \C(\L \cap \L^{\prime})$. Following these observations, finally we note that
	\begin{align*}
	\mu_\C(\L)+ \mu_\C(\L^{\prime})&=\dim\C(\L) + \dim\C(\L^{\prime})\\
	&=\dim(\C(\L)+\C(\L^{\prime})) + \dim(\C(\L) \cap \C(\L^{\prime}))\\
	& \leq \dim(\C(\L \cap \L^{\prime})) + \dim(\C(\L + \L^{\prime}))\\
	&=\mu_\C(\L \cap \L^{\prime}) + \mu_\C(\L + \L^{\prime}).
	\end{align*}
	Therefore, Lemma \ref{27} implies that $\rho_\C(\L + \L^{\prime})+ \rho_\C(\L \cap \L^{\prime}) \leq \rho_{\C}(\L)+\rho_{\C}(\L^{\prime})$.
	\end{enumerate}
	Thus it completes the proof that $\rho_\C$ is a rank function on $\P(\K^{\n})$ and hence we conclude that $(\K^{\n}, \rho_\C)$ is indeed a \emph{sum-matroid}.
\end{proof}
We call the sum-matroid $(\K^{\n},\, \rho_{\C})$ of Theorem \ref{54} to be the \emph{sum-matroid associated to the sum-rank metric code $\C$} and we denote this sum-matroid by $M_{\C}$.

\textbf{Note}: For the rest of the paper, whenever we write $M_\C$, we mean the sum-matroid associated to the sum-rank metric code $\C$.
\begin{cor}\label{cor:::1}
	The rank of the sum-matroid $M_{\C}$ is $\dim\C^\perp$.
\end{cor}

\begin{proof}
Clearly, $\C(\K^\n) = \{\0\}$. Hence in view of \eqref{eq::3}, $\C_{\K^\n} \simeq \C^\perp$ and so $\text{rank}(M_\C)$ = $\rho_\C(\K^{\n}) = \dim \C_{\K^{\n}} = \dim \C^\perp.$
\end{proof}

The above construction of sum-matroids from sum-rank metric codes can be written alternatively as follows.
Let $\C$ be a $[n,k;n_1,\dots,n_{\ell}]$-sum-rank metric code over $(\F;\K_1,\dots,\K_{\ell})$. Suppose that $\H=[\H_1|\cdots|\H_{\ell}]$ is a parity check matrix of $\C$ such that $\H_i$ has $n_i$ columns. 

 For $\L = (\L_1,\dots,\L_{\ell}) \in \PKn$, we define
\begin{equation}\label{eq:5}
	 \tilde{\rho}_\C(\L) = \dim_{\F} \bigoplus_{i=1}^{\ell} \< \H_i \x^T : \x\in \L_i \>_{\F},
\end{equation}
where $\< \H_i \x^T : \x\in \L_i \>_{\F}$ is the $\F$-linear space generated by the vectors $\H_i \x^T$ where $\x$ runs through the elements of $\L_i$.
	Now note that for all $\L \in \PKn$, $\tilde{\rho}_\C(\L) = \rho_\C(\L)$ since both are equal to rank($\H \AA^T$), where $\AA_i \in M_{ N_i \times n_i}(K_i)$ is a generator matrix of $\L_i$ for $1 \leq i \leq \ell$ and $\AA=\diag(\AA_1,\ldots,\AA_{\ell})$.

This construction of sum-matroids from sum-rank metric codes generalizes the constructions of matroids (when  $\n = (1,\ldots, 1)$) and $q$-matroids (when $\ell=1$) in Section \ref{Sec:2.2}.

 The advantage of the construction of $\tilde{\rho_{\C}}$ is that it captures the idea that the rank function $\rho$ of sum-matroids are abstraction of the notion of dependence in vector spaces. Indeed, the rank function $\tilde{\rho}_\C$ in Equation \eqref{eq:5} measures the linear dependency over $\F$ of vectors which are coming from blocks of $K_i$-subspaces.

%Indeed, since the generators in each component $\< \H_i \x^T : \x\in \L_i \>_{\F}$ of the sum in Equation \eqref{eq:5} are from $K_i$-subspaces, but the rank function $\tilde{\rho}_\C$ measures their linear dependency over $\F$.

The following result gives a natural connection between the duality of sum-rank metric codes and the duality of sum-matroids.

\begin{thm}\label{210}
 Let $\rho_\C$ be the rank function of the sum-matroid associated to a sum-rank metric code $\C$. Then the rank function of the dual of the sum-matroid is the same as the rank function of the sum-matroid associated to the dual of $\C$, i.e., $\rho^*_\C = \rho_{\C^\perp}$. 
\end{thm}

\begin{proof}

	Let $\L \in \PKn$. From Definition \ref{def@16}, $\rho^*_\C(\L) = \rho_{\C}(\L^{\perp}) + \Rk(\L) - \rho_{\C}(\K^{\n})$.
	Considering the restriction of the $\F$-linear map $\Pi_{\L}$ on $\C$, we get
	\begin{align}\label{eq::37}
	\dim({\C}^{\perp}_{\L}) &= \dim(\Pi_{\L}(\C)) = \dim(\C)-\dim(\C \cap \Nu_{\L^\perp}).
	\end{align}
	Suppose $G$ and $G_1$ are the generator matrices of $\C$ and $\Nu_{\L}$ of rank $k$ and $k_1$, respectively. Consider the following linear maps
	\begin{align*}
	\Phi: &\Nu_{\L} \longrightarrow \F^k \quad \quad\quad \quad \Psi: \C \longrightarrow \F^{k_1}\\
	      &x \mapsto xG^{T}, \quad \quad \quad \quad \quad \quad x \mapsto xG_1^{T}.
    \end{align*}
    Observe that $\dim(\C) - \dim(\C \cap \Nu_{\L^\perp}) = \dim(\Psi(\C)) = \text{rank}(GG_1^T) = \\ \text{rank}(G_1G^T)= \dim(\Phi(\Nu_{\L})) = \dim(\Nu_{\L}) - \dim(\Nu_{\L} \cap \C^\perp)$. Now combining it with Equations \eqref{eq::37} and \eqref{eq:::2}, we obtain 
    \begin{equation}
	\dim({\C}^{\perp}_{\L}) = \Rk(\L) - \dim({\C^\perp} \cap \Nu_{\L}).
	\end{equation}
	
	Now since $\rho_{\C}(\K^{\n}) = \dim{\C^\perp}$ and $\dim({\C^\perp} \cap \Nu_{\L}) = \dim\C(\L^{\perp})$ (from Equation \eqref{rem}), we see that
	\begin{align*}
	\rho^*_\C(\L) &= \rho_{\C}(\L^{\perp}) + \Rk(\L) - \rho_{\C}(\K^{\n})\\
	&= \rho_{\C}(\L^{\perp}) + \dim({\C}^{\perp}_{\L}) + \dim({\C^\perp} \cap \Nu_{\L}) - \dim{\C^\perp}\\
	&= \dim({\C}^{\perp}_{\L}) + \dim\C(\L^{\perp}) - (\dim{\C^\perp} - \rho_{\C}(\L^{\perp}))\\
	&= \dim(\C^\perp_{\L}) \\
	&= \rho_{\C^\perp}(\L).
	\end{align*}
	 Thus it proves that $\rho_\C^* = \rho_{\C^\perp}$.
\end{proof}
 
 The version of the previous theorem for the particular cases of matroids and $q$-matroids can be found in \cite{BJMS12} and \cite{JP2016}, respectively.

\section{Generalized weights of sum-matroids}\label{Sec:5}
This section deals with the extension of the notion of generalized weights for sum-matroids so that the sum-rank weights of sum-rank metric codes appear as a special case. We prove the monotonicity and Wei-type duality theorem for the generalized weights of a sum-matroid and these can be seen as generalization of the analogous results for matroids and $q$-matroids.
 In this section also we use $\n$ and $\K$ to denote the $\ell$-tuples $(n_1,\ldots,n_{\ell})$ and $(K_1,\ldots,K_{\ell})$, respectively and $\K^{\n}$ to  denote the $\ell$-tuple of vector spaces $(K_1^{n_1},\ldots,K_{\ell}^{n_\ell})$.

We have seen in Section \ref{Sec:2}, the notions of generalized weights of linear codes embedded with the Hamming and the rank metric are generalized to those of matroids (Definition \ref{defn:22}) and $q$-matroids (Definition \ref{defn:28}), respectively. %
 Here we define the generalized weights of sum-matroids using the nullity function in an analogous manner.
 
 Given a sum-matroid with rank function $\rho$, the \e{nullity function} $\nn$ is defined by $$\eta(\L) = \Rk(\L) - \rho(\L),$$
 for $\L \in \PKn$.
 
 \begin{defn}[Generalized weights of a sum-matroid]\label{def@21}
 Let $M= (\K^{\n}, \rho)$ be a sum-matroid with nullity function $\eta$. The $i$-th generalized weight of $M$ is defined by
 	\[
 	d^S_i(M) = \min\{ \Rk(\L) : \L\in \P(\K^{\n})\text{ and } \eta(\L) = i\},
 	\] for $i = 1, \ldots, \nn(\K^n)$.
 \end{defn}

It is clear that the particular cases of $\n=(1,\ldots,1)$ and $\ell = 1$ recover the Definition \ref{defn:22} of generalized weights of matroids and the Definition \ref{defn:28} of generalized weights of $q$-matroids, respectively. 
%Let $M=(E,\rho)$ be a $q$-matroid. The $i$-th generalized weights of the $q$-matroid $M$ with nullity function $\nn$ is defined as
%\[
%d_i^R(M) = \min\{ \dim_{\Fq} U\colon U\in \P(E) \text{ and } \nn (U)= i \},
%\]
%for $i=1, \dots, \nn(E)$.
 %In \cite{GJ19}, the generalized weights of \linebreak $(q,m)$-polymatroids, which corresponds to (matrix) rank metric codes, are defined. But any vector rank metric code can be seen also as a (matrix) rank metric code. Hence generalized weights of $q$-matroids are just particular cases those for $(q,m)$-polymatroids. In what follows, whenever we mention the results from \cite{GJ19}, it will be understood that these are for the special class of Delsarte rank metric codes which arise from vector rank metric codes or $(q,m)$-polymatroids with $m=1$. The case $m=1$ is further studied in \cite{GJLR19}.
 
Analogous to the classical matroid case, here is an alternative expression for $d^S_i(M)$.
 \begin{lem}\label{211}
  Let $M= (\K^{\n}, \rho)$ be a sum-matroid with nullity function $\eta$. Then for $i = 1, \ldots, \nn(\K^n)$,
\begin{align*} & \min\{ \Rk(\L) : \L\in \P(\K^{\n}) \text{ and } \eta(\L) = i\} \\
 & \quad \quad \quad= \min\{ \Rk(\L) : \L\in \P(\K^{\n}) \text{ and } \eta(\L) \geq i\}.
 \end{align*}
 \end{lem}

 \begin{proof}	
For any $i \in \{1, \ldots, \nn(\K^{\n})\}$, we show that if $\min\{ \Rk(\L) : \L\in \P(\K^{\n}) \text{ and } \nn(\L) \geq i\} = \Rk(\L_1)$ for some $\L_1 \in \P(\K^{\n})$ with $\nn(\L_1) \geq i$, then $\nn(\L_1) = i$. Now since $\Rk(\L_1) \geq \nn(\L_1 ) \geq i \geq 1$, we choose $\L_2 \subsetneq \L_1$, i.e., $\Rk(\L_2) = \Rk(\L_1) -1$. Then by minimality of $\Rk(\L_1)$ it follows that $\nn(\L_2) < i $ or $\nn(\L_2) \leq i-1$. Now using the definition of nullity function and the property (R\ref{R2}) of  $\rho$, we get $\nn(\L_1) = \Rk(\L_1) - \rho(\L_1) \leq \Rk(\L_1) - \rho(\L_2) = \Rk(\L_2) +1 - \rho(\L_2) = \nn(\L_2) +1$. Thus $\nn(\L_1) \leq \nn(\L_2) + 1$ implies that $i \leq \nn(\L_1) \leq i$ and therefore $\nn(\L_1) = i$.
This completes the proof.
 \end{proof}

The following result shows that the notion of generalized weights of sum-matroids extends the generalized sum-rank weights of sum-rank metric codes.
 \begin{thm}\label{29}
	Let $\C$ be an $[n,k;n_1,\dots,n_{\ell}]$-sum-rank metric code and let $M_\C$ be the sum-matroid associated to $\C$. Then for $i=1,\ldots,k$,
	\[
	d^S_i(M_\C) = d_{SR,i}(\C).
	\]
\end{thm}

 \begin{proof}
	Let $\nn_\C$ denote the nullity function of $M_\C$.
	From Definition \ref{def::12} and Lemma \ref{211}, it is clear that to prove the above theorem, it is enough to show that $\dim(\C \cap \Nu_{\L}) = \eta_{\C}(\L)$. By definition, $\eta_{\C}(\L) = \Rk(\L) - \rho_{\C}(\L)$. Applying Theorem \ref{210} to $\C^\perp$ we see that, $\rho_{\C}(\L) = \rho_{\C^\perp}(\L^\perp) + \Rk(\L) - \rho_{\C^\perp}(\K^{\n})$. Therefore, in view of \eqref{eq::3} and Corollary \ref{cor:::1}, $\eta_{\C}(\L) =  \rho_{\C^\perp}(\K^{\n}) - \rho_{\C^\perp}(\L^\perp) = \dim_{\F}{\C} - \dim \C^{\perp}_{\L^{\perp}}=\dim \C^\perp({\L}^{\perp}) = \dim(\C \cap \Nu_{\L}).$ Hence we have the desired result.
\end{proof}

Next we prove the monotonicity property of the generalized weights of sum-matroids which generalizes the analogous results for matroids and $q$-matroids as stated in Proposition \ref{pro:mono} of Section \ref{Sec:2.2}.
%Similar results for matroids and $(q,m)$-polymatroids can be found in \cite{BJMS12} and \cite{GJ19}. 

\begin{pro}[Monotonicity] Let $M=(\K^{\n}, \rho)$ be a sum-matroid and $d_i =  d^S_i(M)$ is the $i$-th generalized weights of $M$ for $1 \leq i \leq \nn(\K^{\n})$. Then $d_1 < d_2 < \cdots < d_{\nn(\K^{\n})}$. 
\end{pro}

\begin{proof}
Fix $i \in \{1, \ldots, \nn(\K^{\n})-1\}$. From Definition \ref{def@21} and Lemma \ref{211}, it is obvious that $d_{i+1}\geq d_i$. Let $\L \in \P(\K^{\n})$ such that $\Rk(\L) = d_{i+1}$ and $\nn(\L)=i+1$. We want to show that there is an $\L' \in \PKn$ with $\L' \subseteq \L$ such that $\eta(\L') \geq i$ and $\Rk(\L')= \Rk(\L) -1$. As $\Rk(\L) \geq \nn(\L) \geq 2$, we can always choose $\L' \subsetneq \L$	with $\Rk(\L^{\prime})= \Rk(\L)-1$.  
	
	Now note that $i+1 = \eta(\L) = \Rk(\L) - \rho(\L) \leq \Rk(\L') +1 - \rho(\L') =  \eta(\L') +1$. So $\eta(\L') \geq i$, i.e., $d_i \leq \Rk(\L') = d_{i+1} - 1$. Therefore we obtain $d_i < d_{i+1}$ for any $1 \leq i  \leq \nn(\K^{\n})-1$.
\end{proof}	

Now we give a Wei-type duality theorem for sum-matroids which generalizes all the analogous results for codes as in Theorem \ref{thm@5} and for matroid structures in Theorem \ref{Wei:matroids}.  
 \begin{thm}[Wei-type duality]\label{thm@55}
	Let $M=(\K^{\n}, \rho)$ be a sum-matroid with rank $k$ and $M^*=(\K^{\n}, \rho^*)$ be its dual. Also let $n = n_1 + \dots + n_{\ell}$. If we let $d_i =  d_i^S(M)$ for $i = 1, \ldots, n-k$ and  $d_j^{\perp} =  d_j^S(M^*)$ for $j =1, \ldots, k$, then it holds
	that\[
	\{1, 2, . . . , n\} = \{d^{\perp}_1, \ldots , d^{\perp}_k \}\cup\\
	\{n + 1 - d_1, \ldots, n+1-d_{n-k} \},
	\]
	where the union is disjoint. In particular, the generalized weights of a sum-matroid $M$ uniquely
	determine those of $M^*$. 
\end{thm}

\begin{proof}
	It is sufficient to show that for a fixed $r$ with $1 \leq r \leq n-k$,
	\begin{enumerate}[(1)]
		\item $d_t^{\p} \leq n - d_r$ for $t=k+r-d_r$,
		\item For every $s>0$, $d_{t+s}^{\perp} \neq n-d_r +1$ where $t = k+r-d_r$.
	\end{enumerate}
	If we assume these to be true, then we have $d_t^{\p} < n-d_r+1$ and $d_{t+s}^{\p} \neq n-d_r+1$ $\forall$ $s > 0$ and for $t =k+r-d_r$. So $ n-d_r+1 \notin \{d_1^{\p}, \ldots, d_k^{\p}\}$, for all $1 \leq r \leq n-k$. As all these $d_i$ and $d_i^{\p}$ are positive integers less or equal to $n$, the monotonicity property implies the desired result.
	
	Now to prove (1), first we fix $r$ such that $1 \leq r \leq n-k$. Let $\nn$ (resp. $\nn^*$) denote the nullity function of $M$ (resp. $M^*$). Let $\L \in \PKn$ be such that $d_r = \Rk(\L)$ and $ \eta(\L)\geq r$. Then
\begin{align*}
\eta^*(\L^\perp)&= \Rk(\L^\perp)-\rho^*(\L^\perp)\\
&= \Rk(\L^\perp)-\rho(\L)-\Rk(\L^\perp)+\rho(\K^{\n}).
\end{align*}
	Hence
	\[
	d_r-k+\eta^*(\L^\perp) = \Rk(\L)-\rho(\K^{\n})+\Rk(\L^\perp)-\rho(\L)-\Rk(\L^\perp)+\rho(\K^{\n})=\eta(\L).	
	\]
	But $\eta(\L) \geq r$, this implies that $\eta^{*}(\L^{\p}) \geq r+k-d_r=t$. Since $\Rk(\L^{\p}) = n-d_r$, it follows that $d_t^{\p} \leq n-d_r$.
	 
	 Now we prove $(2)$. For some $s > 0$, assume that $d_{t+s}^{\p} = n-d_r+1$. Let $\L \in \PKn$ be such that $d_{t+s}^{\p} = \Rk(\L)$. By definition, $\eta^{*}(\L) \geq t+s$ and in particular, $\eta^{*}(\L) > t$.
	Now 
	\begin{align*}
		&\Rk(\L)-\rho^{*}(\L) > t\\
	    \Rightarrow \quad&n-d_r+1 - \rho(\L^{\p}) - \Rk(\L) +\rho(\K^{n}) >t \\
		\Rightarrow \quad&n-d_r+1 - \rho(\L^{\p}) - n+ \Rk(\L^{\p}) +\rho(\K^{n}) > t\\
		\Rightarrow \quad&\eta(\L^{\p}) - d_r+1+k >t\\
		\Rightarrow \quad&\eta(\L^{\p}) >t+d_r-1-k\\
		\Rightarrow \quad&\eta(\L^{\p})>r-1\\
		\Rightarrow \quad&\eta(\L^{\p})\geq r.
	\end{align*}
	So $d_r \leq \Rk(\L^{\p}) = n- d_{t+s}^{\p} = d_r -1$, which is not possible. Hence our assumption that $d_{t+s}^{\p} = n-d_r+1$ is not true. Thus $(2)$ is proved and it completes the proof of the theorem.
	
\end{proof}

The following is the definition of uniform sum-matroids, which is a generalization of uniform matroids \cite{Ox06} and uniform $q$-matroids \cite{JP2016}. 
\begin{defn}
	Let $k$ be an integer $\leq n$. The uniform sum-matroid $U_{k,n}$ is defined as $(\K^{\n}, \rho)$, where $\rho(\L) = \Rk(\L)$ for all $\L \in \P(\K^{\n})$ with $\Rk(\L) \leq k$, while $\rho(\L) = k$ if $\Rk(\L)>k$.
\end{defn}

Using the following characterization of $r$-MSRD codes given in \cite{MP2019}, we show that uniform sum-matroids correspond to the MSRD codes given in Definition \ref{def@13}. 

 \begin{thm}[\cite{MP2019}]
	Given any integers $r,k$ with $1 \leq r \leq k$, a $[n,k;n_1,\dots,n_{\ell}]$- linear code $\C \subseteq \F^n$ is $r$-MSRD
	if and only if $\C \AA \subseteq \F^n $ is $r$-MDS, for all $\AA = \diag(\AA_1, \ldots , \AA_{\ell}) \in M_{n \times n}(\F)$ such that $\AA_i \in GL_{n_i}(K_i)$  for $i = 1, \ldots, \ell$.
\end{thm}

\begin{exa}
Let $\C$ be a $[n,k;n_1,\dots,n_{\ell}]$-MSRD code. From the above result $\C \AA$ is MDS, for all  $\AA = \diag(\AA_1, \ldots , \AA_l) \in \F^{n \times n}$ such that $\AA_i \in K_i^{n_i \times n_i}$ is invertible, for $i = 1, \ldots, l$. 
We show that the sum-matroid corresponding to $\C$ is a uniform sum-matroid $U_{n-k,n}$ and the corresponding dual sum-matroid is also a uniform sum-matroid $U_{k,n}$. Let $M_{\C}$ be the corresponding sum-matroid. 
Let $\H \in M_{n-k \times n}(\F)$ be a parity check matrix of $\C$. Suppose $\L = (\L_1, \ldots, \L_l) \in \P(\K^{n})$ has rank $N$, where $\dim_{K_i}(\L_i)=N_i$ for $i=1,\ldots,l$ and $N=N_1+\dots+N_l$. 
Consider $\G_i \in M_{N_i \times n_i}(\K_i)$ to be a generator matrix of $\L_i$ for $i=1,\ldots,l$. Then we can extend $\G_i$ to invertible matrices $\tilde{\G}_i \in M_{n_i \times n_i}(\K_i)$ for $i=1,\ldots,l$. 
Consider $\AA = \diag(\tilde{\G}_1^T,\ldots,\tilde{\G}_l^T)$. It is clear that $\H\AA \subseteq \F^n$ generates an $[n,n-k]$ MDS code over $\F$ from the previous theorem. 
From a characterization of MDS code (see Theorem \ref{thm::24} for example) we know that any set of $n-k$ columns of $\H\AA$ is linearly independent over $\F$. Now if $\Rk(\L) = N \geq n-k$, then $\rho_{\C}(\L)= \dim_{\F}(\C_{\L}) = n-k$ and for $N<n-k , \, \rho_{\C}(\L)= \dim_{\F}(\C_{\L}) = N$. Hence $M_{\C}$ is a uniform sum-matroid of rank $n-k$. The dual sum-matroid of rank $k$ is also a uniform sum-matroid of rank $k$ as it is coming from the dual of an MSRD code which is again MSRD (See \cite[Thm. 5]{MP2019}).
\end{exa}  
\begin{rem}
The correspondence between MSRD codes and uniform sum-matroids is the extension of the correspondence between MDS (resp. MRD) codes and uniform matroids (resp. $q$-matroids).
\end{rem}

\section{Conclusion and further comments}\label{Sec:6}

 We have defined sum-matroids in terms of rank functions. The rank function of sum-matroids generalizes the rank function of matroids and $q$-matroids. So it is  natural to ask which results in matroid theory will have sum-matroid analogue. We have answered this question partially. Duality, generalized weights, which are important invariants from coding theoretic point of view, have been defined more generally for sum-matroids. There are many cryptomorphic definitions of $q$-matroids \cite{JP2016,BCR21} which are $q$-analogues of those for matroids \cite{Whi86}. One can expect to find analogous cryptomorphic definitions for sum-matroids also. 
%But there are still a lot of problems regarding sum-matroids to research on. 
%For instance, the discussion on further research directions in the concluding section of \cite{JP2016} is valid for sum-matroids also. 

Further, the rank generating functions for sum-matroids can be studied. More importantly, as noted in \cite{JP2016}, matroids and $q$-matroids are particular cases of a more general structure of modular lattices which are complemented and some work about them were done by Crapo in \cite{Cra64}. Sum-matroids also appear to be complemented lattices and hence we can get hints from \cite{Cra64} and further develop the theory of sum-matroids. 
Matroids induced from matrices are called representable matroids. As we mentioned before, there are matroids which are non-representable (See, e.g., the V{\'a}mos matroid $V_8$ \cite{Ox06}). Constructions of non-representable $q$-matroids also exist (See \cite{BCR21,CJ21}). So one might ask if there is any non-representable sum-matroid which is neither a matroid nor a $q$-matroid.

As we mentioned in the introduction, one can associate a Stanley-Reisner ring to a matroid and the Betti numbers of the ring are related to the generalized weights of a linear code if the matroid is constructed from that code. In fact, for constant weights codes and more generally, for linear codes with a pure resolution, the generalized weights and Betti numbers define each other \cite{JV13,JPV21}. We want to find a similar result for sum-matroids but anyhow, a study about the constant sum-rank weight linear codes is a good step toward this. 

\subsection*{\centering Acknowledgment}	
We thank Prof. Sudhir Ghorpade for reading our work and giving many useful comments that helped us to make approapriate revisions in this version. 
	
%{\footnotesize \bibliography{latestreferences}}

\end{document}